\documentclass[11pt,a4paper]{article}
\usepackage{amssymb,amsmath}

\newtheorem{theorem}{Theorem}[section]
\newtheorem{corollary}[theorem]{Corollary}
\newtheorem{definition}[theorem]{Definition}
\newtheorem{lemma}[theorem]{Lemma}
\newtheorem{proposition}[theorem]{Proposition}
\newtheorem{remark}[theorem]{Remark}

\newtheorem{question}[theorem]{Question}
\newtheorem{example}[theorem]{Example}

\newenvironment{proof}{\begin{trivlist}\item[]{\it
Proof.}}{\hfill$\square$\end{trivlist}}

\newenvironment{proofofthm}[1]{\noindent{\it Proof of Theorem 
#1}}{\hfill$\square$\\\mbox{}}

\def\mc{{\mathbb{C}}}
\def\pol{{\mathrm{Pol}}}
\def\sltwo{{\mathrm{SL}}_2({\mathbb{C}})} 
\def\gltwo{{\mathrm{GL}}_2({\mathbb{C}})} 
\def\mz{{\mathbb{Z}}}
\def\mr{{\mathbb{R}}}
\def\mn{{\mathbb{N}}}
\def\limzr{{\mathrm{lim}}_{z\to 0}\rho(z)}
\def\onepsg{1{\mathrm{-PS}}}
\def\mds{{\mathrm{mds}}}

\begin{document}
\title{Helly dimension of algebraic groups} 
\author{M\'aty\'as Domokos
and Endre Szab\'o  
\thanks{Both authors are partially supported by 
OTKA NK72523 and K61116. }
\\ 
\\
{\small R\'enyi Institute of Mathematics, Hungarian Academy of 
Sciences,} 
\\ {\small P.O. Box 127, 1364 Budapest, Hungary,} 
%{\small E-mail: domokos@renyi.hu } 
}
\date{}
\maketitle 
\begin{abstract} It is shown that for a linear algebraic group $G$ over a field of characteristic zero, there is a natural number $\kappa(G)$ such that if 
a system of Zariski closed cosets in $G$ has empty intersection, then there is 
a subsystem consisting of at most $\kappa(G)$ cosets with empty intersection. 
This is applied to the study of algebraic group actions on product varieties. 
\end{abstract}


\section{Introduction}\label{sec:intro} 

Recall the following result on the invariant theory of binary forms from \cite{d:sep}: 
If two points $(v_1,\ldots,v_m,w)$ and $(v_1',\ldots,v_m',w')$ in 
$V^m\oplus W$ can be separated by polynomial $\sltwo$-invariants 
(here $W,V$ are finite dimensional polynomial $\sltwo$-modules, 
and $V$ is irreducible), then there are at most seven indices 
$\{i_1,\ldots,i_d\}$ $(d\leq 7)$ such that $(v_{i_1},\ldots,v_{i_d},w)$ 
and  $(v_{i_1}',\ldots,v_{i_d}',w')$ can be separated by polynomial 
$\sltwo$-invariants on $V^d\oplus W$. 
The present paper explains why a finite bound (seven) independent from $V,W$ exists here. 
In particular, we obtain a far reaching generalization of this fact. 

Roughly speaking, it turns out that the above mentioned finite bound has two components: 
one is the finiteness of the Helly dimension of the algebraic group 
$\sltwo$, the other is finiteness of a number $\delta(\sltwo)$ defined in terms of closed orbits in products of $\sltwo$-varieties. 

The Helly dimension $\kappa(G)$ of 
a finite group $G$ was introduced in \cite{d:sep} as the minimal natural number $d$ having the following property: if a system of cosets in $G$ has no common element, then one can find $\leq d$ cosets in the system with empty intersection. (The name "Helly dimension" was chosen because of the analogy 
with a theorem of Helly in convex geometry, see \cite{steinitz}.) 
We extend this definition to an arbitrary topological group  by restricting to closed subgroups (hence closed cosets): 

\begin{definition}\label{def:hellydim} Given a topological group $G$, its {\it Helly dimension} 
$\kappa(G)$ is the minimal natural number $d$ such that any finite system of closed cosets in $G$ with empty intersection has a subsystem consisting of at most $d$ cosets with empty intersection, and setting $\kappa(G)=\infty$ if there is no natural number $d$ with the above property. 
\end{definition} 

(Our restriction to only consider finite systems of cosets plays an essential role in this definition, see 
Example~\ref{example:infinitecosets}.)   
In the present paper we mainly work in the context of algebraic groups. 
An algebraic group $G$ is viewed as a topological group with the Zariski topology, and its Helly dimension  $\kappa(G)$ is understood accordingly. 
One of our main results is that if $G$ is a linear algebraic group over an algebraically closed field of characteristic zero, then $\kappa(G)$ is finite 
(see Theorem~\ref{thm:alggphelly}). 
In Section~\ref{sec:abel} we determine the Helly dimension of finite abelian groups. The finiteness of $\kappa(G)$ is deduced from this with the aid of 
Platonov's Lemma and Jordan's Theorem on finite linear groups. 

We have only partial results about the finiteness of the quantity $\delta(G)$ (cf. Definition~\ref{def:delta}): namely, we determine $\delta(G)$ when $G$ is a torus, or $\sltwo$, or $\gltwo$ (see Proposition~\ref{prop:torus}, Theorems~\ref{thm:sl2}, \ref{thm:gl2}). 
We mention that the proof for tori reduces to a Carath\'eodory type theorem, another fundamental result from convex geometry. 

Whereas the main theme of the present paper is to bound the number of variables necessarily involved in a separating system of invariants in terms of the group, 
\cite{d:sep} contains general results on similar bounds in terms of the dimension of the irreducible components of the representation. 
In Section~\ref{sec:separating} we make sharp a general result from 
\cite{d:sep} in this direction. 

The naturality of the concept of the algebraic Helly dimension is underlined 
in Section~\ref{sec:rational} by some applications to rational invariants of algebraic groups. In particular, as a corollary of the finiteness of $\kappa(G)$ in characteristic zero, we obtain a rather elementary and general statement on fields of rational invariants (see Theorem~\ref{thm:rational}), that has no analogues in positive characteristic. 

Finally, in Section~\ref{sec:lie} we point out that 
compact real Lie groups have finite Helly  dimension. 

%%%%%%%%%%%%%%%%%%%%%%%%%%%%%%%%%%%%%%%%%%%%%%%%%%%%%%%%

\section{Abelian groups}\label{sec:abel}

\begin{lemma}\label{lemma:AxB} 
Let $A,B$ be finite abelian groups of coprime order. 
Then 
\[\kappa(A\times B)=\max\{\kappa(A),\kappa(B)\}.\] 
\end{lemma} 

\begin{proof} The inequality $\geq$ is obvious, since $A$ and $B$ are subgroups of $A\times B$. 
For the reverse inquality, note that the assumption on the orders of $A,B$ 
implies that any subgroup in $A\times B$ is of the form $A_0\times B_0$ 
with $A_0\leq A$, $B_0\leq B$. 
Take cosets $(a_i,b_i)\cdot A_i\times B_i$ $(i=1,\ldots,m)$ in $A\times B$, 
and assume that any $d$ of them have a non-empty intersection, where 
$d=\max\{\kappa(A),\kappa(B)\}$. 
It follows that any $d$ of the cosets $a_iA_i$ have a non-empty intersection, 
implying (by the choice of $d$) that there is a common element $a\in A$ of all these cosets. 
Similarly there is a common element $b\in B$ in the intersection of the cosets 
$b_iB_i$. Clearly, $(a,b)$ is contained in the intersection of the given $m$ cosets in $A\times B$. 
\end{proof} 

\begin{lemma}\label{lemma:p-group} 
Let $p$ be a prime, and $P$ a finite abelian $p$-group minimally generated by $d$ elements. Then $\kappa(P)=d+1$. 
\end{lemma} 

\begin{proof} 
For a subgroup $Q$ in $P$, set ${\mathrm{soc}}(Q)=\{g\in Q\mid g^p=1\}$. 
Clearly ${\mathrm{soc}}(Q)$ is an elementary abelian $p$-group, whence can be viewed as a vector space of dimension $s$ over the field $\mathbb{F}_p$ of $p$ elements, where $s$ is the minimal number of generators of $Q$. In particlular, ${\mathrm{soc}}(P)\cong \mathbb{F}_p^d$. 
Clearly ${\mathrm{soc}}(Q)={\mathrm{soc}}(P)\cap Q$, and if ${\mathrm{soc}}(Q)$ is trivial, then $Q$ is the trivial subgroup. 
Note also that if $Q_1,Q_2$ are subgroups of $P$, and $a_1,a_2\in P$, then the intesection of the cosets $a_1Q_1$ and $a_2Q_2$ is either empty or a coset with respect to $Q_1\cap Q_2$. 

Let $C_1,\ldots,C_m$ ($m\geq d+1$) be cosets in $P$ with empty intersection, $C_i=a_iP_i$, where 
$a_i\in P$ and $P_i$ is a subgroup of $P$. We may assume that $P_i\neq P$. 
After a possible renumbering of the cosets, we may assume that  for some $e\leq d$ we have 
$$\dim_{\mathbb{F}_p}(\cap_{i=1}^e{\mathrm{soc}}{P_i})=\dim_{\mathbb{F}_p}(\cap_{i=1}^m{\mathrm{soc}}{P_i}). $$ 
If $\cap_{i=1}^e{\mathrm{soc}}{P_i}=\{0\}$, then $\cap_{i=1}^eP_i=\{0\}$, hence 
either $\cap_{i=1}^eC_i$ is empty and we are done, or 
$\cap_{i=1}^eC_i$ consists of a single element $a$. In the latter case there is a $j>e$ such that 
$a\notin C_j$, so the $e+1$ cosets $C_1,C_2,\ldots,C_e,C_j$ have an empty intersection. 

If $S:=\cap_{i=1}^e{\mathrm{soc}}{P_i}\neq\{0\}$, then $S\leq P_j$ for all $j=1,\ldots,m$. Consider the natural surjection $\eta:P\to P/S$, and denote by $\overline{C_i}$ the image of $C_i$. We have that 
$C_i=\eta^{-1}(\eta(C_i))$, implying that $\cap_{i=1}^m\overline{C_i}=\emptyset$. 
By induction on the size of $P$ we may assume that $\kappa(P/S)\leq d$, hence there are $d+1$ cosets, say $\overline{C_1},\ldots,\overline{C_{d+1}}$ with empty intersection. Then $\cap_{i=1}^{d+1}C_i=\emptyset$. 

To show the reverse inequality $\kappa(P)\geq d+1$, note that 
$L_i=\{(x_1,\ldots,x_d)\in\mathbb{F}_p^d\mid x_i=0\}$, $i=1,\ldots,d$, 
and $L_0=\{(x_1,\ldots,x_d)\mid \sum x_i=1\}$ are $d+1$ cosets in 
$\mathbb{F}_p^d={\mathrm{soc}}(P)$, whose intersection is empty, but the intersection of 
any $d$ of them is non-empty. 
\end{proof} 
 
\begin{corollary}\label{cor:abelkappa} Let $A$ be a finite abelian group minimally generated by $d$ elements. 
Then $\kappa(A)=d+1$. 
\end{corollary}

\begin{proof} The structure theorem of finite abelian groups implies that if 
$p_1,\ldots,p_r$ are the distinct prime divisors of $|A|$, then 
$A\cong P_1\times \cdots\times P_r$, where $P_i$ is a $p_i$-group generated by 
$\leq d$ elements for $i=1,\ldots,r$, 
and some $P_i$ can not be generated by $d-1$ elements. 
Therefore by Lemma~\ref{lemma:p-group} we have $\kappa(P_i)\leq d+1$  
for $i=1,\ldots,r$, with equality for some $i$, whence 
$\kappa(A)=d+1$ by Lemma~\ref{lemma:AxB}. 
\end{proof} 

We shall make use of the following general lemma: 

\begin{lemma}\label{lemma:extension} 
Let  $H$ be a normal subgroup in the finite group $G$. 
Then 
$$\kappa(G)\le\kappa(H)\cdot \kappa(G/H).$$
\end{lemma} 

\begin{proof}
Let ${\mathcal{C}}$ be a finite collection of left cosets in $G$,
  and assume, that any
  $\kappa(H)\kappa(G/H)$ of them have non-empty intersection.
  Let ${\mathcal{D}}$ be the collection of all $\kappa(H)$-tuple intersections
  of members of ${\mathcal{C}}$, and $\overline{\mathcal{D}}$ be the collection of their
  images in $G/H$. Then any $\kappa(G/H)$ members of $\overline{\mathcal{D}}$
  have nonempty intersection, hence all of them have a common element
  $xH$ for some $x\in G$. Hence each member of ${\mathcal{D}}$ intersects $xH$.
  Hence any $\kappa(H)$ members of ${\mathcal{C}}$ have a common point inside
  $xH$. Therefore in the collection
  $$\left\{x^{-1}C\bigcap H\;\Big|\;C\in{\mathcal{C}}\right\}$$
  any $\kappa(H)$ members have nonempty intersection. But then the
  intersection of all members is nonempty, 
  hence the intersection of all members of $\mathcal{C}$ is nonempty as well.
\end{proof}

Next we state a corollary for the Helly dimension of linear groups. 
By a theorem of Jordan (see for example \cite{bass}), for any natural number $n$ there exists a constant $J(n)$  such that if $G$ is a finite subgroup of the general linear group $GL_n(k)$ over a field $k$ of characteristic zero, then $G$ contains an abelian normal subgroup of index $\leq J(n)$. 

\begin{corollary} \label{cor:linear} 
Let $G$ be a finite subgroup of 
$GL_n(k)$, where $k$ is a field of characteristic zero. 
Then $\kappa(G)\leq (n+1)(1+\log_2(J(n)))$.  
\end{corollary} 

\begin{proof} Let $A$ be an abelian normal subgroup of $G$ such that for 
$H=G/A$ we have $|H|\leq J(n)$. 
Note that after extending $k$ to its algebraic closure, $A$ can be conjugated into the group of diagonal matrices, hence $A$ is generated by $\leq n$ elements. 
Therefore $\kappa(A)\leq n+1$ by Corollary~\ref{cor:abelkappa}. 
Following \cite{d:sep}, denote by $\lambda(H)$ the maximal length of a chain of proper subgroups in $H$. Then we have the inequality 
$\kappa(H)\leq \lambda(H)+1$ (see Lemma 4.2 in \cite{d:sep}), and one has the 
trivial bound $\lambda(H)\leq \log_2(|H|)$. 
By Lemma~\ref{lemma:extension} we conclude 
$\kappa(G)\leq\kappa(A)\kappa(H)\leq (n+1)(1+\log_2(J(n)))$. 
\end{proof}

%%%%%%%%%%%%%%%%%%%%%%%%%%%%%%%%%%%%%%%%%%%

\section{Linear algebraic groups}\label{sec:alggroup} 

In this section $G$ is a linear algebraic group over an algebraically closed 
field $k$. Write $\dim(G)$ for the dimension of $G$ as an algebraic variety over $k$. By a representation of $G$ we mean a 
morphism $G\to GL_n(k)$ of algebraic groups. Denote by $H^{\circ}$ the connected component of the identity in a linear algebraic group $H$. 

\begin{theorem}\label{thm:alggphelly} 
Suppose ${\mathrm{char}}(k)=0$. 
Then the number $\kappa(G)$ is finite: we have 
\[\kappa(G)\leq \dim(G)+(n+1)(1+\log_2(J(n))),\]   
where $n$ is the dimension of a faithful representation of $G$.  
\end{theorem} 

\begin{proof} Note that if the intersection of the cosets $gH$ and $hK$ is non-empty, then it is a coset with respect to $H\cap K$, hence 
$\dim(gH\cap hK)<\min\{\dim(H),\dim(K)\}$, unless $H^{\circ}=K^{\circ}$. 
Set $d:=\dim(G)+(n+1)(1+\log_2(J(n)))$. 
Let $C_1,\ldots,C_m$  (with $m\geq d$) be Zariski closed cosets in $G$ such that any $d$ of them have a common element.  
Denote by $e$ the maximal codimension of a non-empty intersection of cosets $C_i$. After a possible renumbering of the cosets we may assume that  
$\dim(G)\geq \dim(C_1)>\dim(C_1\cap C_2)>\cdots >\dim(C_1\cap \cdots \cap C_e)\geq 0$. 
Replacing the cosets $C_i$ by $g^{-1}C_i$ for some common element $G$ of 
$C_1,\ldots,C_e$, we may assume that $C_1,\ldots,C_e$ are subgroups of $G$. Write $H$ for their intersection. Consider the cosets $D_j:=H\cap C_j$, $j=e+1,\ldots,m$ 
in $H$. Since $e\leq \dim(G)$, by our assumption any $d-\dim(G)$ of these cosets 
have a non-empty intersection. Moreover, any coset $D_j$ is a coset with respect 
to a subgroup of $H$ containing $H^{\circ}$ (otherwise ${\mathrm{codim}}(C_1\cap\cdots\cap C_e\cap C_j)>e$ contradicting to the 
definition of $e$). Write $\eta:H\to H/H^{\circ}$ for the natural surjection. Then $D_j=\eta^{-1}(E_j)$ for some coset $E_j$ of the finite group 
$H/H^{\circ}$. Moreover, any $(n+1)(1+\log_2(J(n)))$ of the cosets $E_j$ have a non-empty intersection. By a lemma of Platonov (see for example \cite{bass}) 
$H/H^{\circ}$ is isomorphic to a subgroup of $H$, hence a subgroup of $G$. 
Therefore by Corollary~\ref{cor:linear}, the intersection 
of $E_{e+1},\ldots,E_m$ is non-empty, implying that the intersection of 
$D_{e+1},\ldots,D_m$ is non-empty. Consequently, the original cosets 
$C_1,\ldots,C_m$ have a non-empty intersection.  
\end{proof} 

\begin{proposition}\label{prop:hellyinfinite} 
Suppose ${\mathrm{char}}(k)=p>0$. 
Then $\kappa(G)$ is finite if and only if $G^{\circ}$ is a torus. 
\end{proposition} 

\begin{proof} 
Let $k_{add}$ denote the additive group of $k$. 
For any $d\in \mn$ take $d$ elements in $k$ that are linearly independent over 
the prime subfield. They generate an elementary abelian $p$-subgroup $P_d$ of 
rank $d$ in $k_{add}$. 
Consequently, 
\[\kappa(k_{add})\geq \kappa(P_d)=d\quad\mbox{ for all }\quad d\in\mn.\]
This implies that as soon as some linear algebraic group $G$ contains $k_{add}$ as a Zariski closed subgroup, then 
$\kappa(G)\geq\kappa(k_{add})=\infty$. 

Suppose now that $G^{\circ}$ is not a torus. Let $B$ be a Borel subgroup of $G$, and denote by $B_u$ the unipotent radical of $B$. Then $\dim(B_u)\geq 1$ 
(see Corollary 11.5 (1) in \cite{borel}), hence 
$B_u$ contains a closed subgroup isomorphic to $k_{add}$ 
(see Theorem 10.6. (2) in \cite{borel}), implying that $\kappa(G)=\infty$. 

Conversely, let $G$ be an $n$-dimensional torus $(k^{\times})^n$. 
Recall that any factor group of $G$ by a closed normal subgroup is also a torus of dimension $\leq n$, and any finite subgroup of $G$ is generated by at most $n$ elements. So any finite subquotient of $G$ is generated by $\leq n$ elements. 
Using this simple fact instead of Jordan's Theorem, the proof of Theorem~\ref{thm:alggphelly} applied for the torus $G$ yields the bound 
$\kappa(G)\leq n+(n+1)=2n+1$. 

Note finally that $\kappa(G)\leq\kappa(G^{\circ})\cdot\kappa(G/G^{\circ})$ 
by Lemma~\ref{lemma:alggpextension} below.  
\end{proof} 

\begin{lemma}\label{lemma:alggpextension} Let $N$ be a Zariski closed normal subgroup 
of the linear algebraic group $G$. We have the inequality 
\[\kappa(G)\leq\kappa(G/N)\cdot\kappa(N).\] 
\end{lemma} 

\begin{proof} Recall that a homomorphism of linear algebraic groups maps a closed subgroup onto a closed subgroup. 
Therefore the proof of Lemma~\ref{lemma:extension} can be repeated verbatim. 
\end{proof} 

%%%%%%%%%%%%%%%%%%%%%%%%%%%%%%%%%%%%%%%%%%%%

\section{An application for rational invariants}\label{sec:rational} 

Given topological $G$-spaces $X_i$, $i=1,\ldots,m$ (i.e. $G$ is a topological group, and the $X_i$ are topological spaces endowed with a continuous $G$-action), we shall 
consider the diagonal action of $G$ on the product set $X=\prod_{i=1}^mX_i$. 
For a subset $I\subseteq \{1,\ldots,m\}$, write $X_I:=\prod_{i\in I} X_i$, and denote  $\pi_I:X\to X_I$, $x\mapsto x_I$ the corresponding projection map. 

The following lemma explains the relevance of the concept of Helly dimension for transformation groups. 

\begin{lemma}\label{lemma:hellydimnatural}  
Let $\mathcal{T}$ be a collection of topological $G$-spaces, and assume that $G/H\in\mathcal{T}$ for all closed subgroups $H\leq G$. 
Then the  Helly dimension $\kappa(G)$ is the minimal natural number $d$ such that 
for any finite product $X=\prod X_i$ (with $X_i\in\mathcal{T}$) and for any points $x,y\in X$ having different $G$-orbits, there is an index subset $I$ of size at most $d$ such that   
the projections $x_I$ and $y_I$ belong to different orbits in $X_I$. 
\end{lemma} 

\begin{proof}  Take $x,y\in X$, and for each $i\in\{1,\ldots,m\}$, consider the subset 
$C_i$ of elements of $G$ moving $x_i\in X_i$ to $y_i\in X_i$. Clearly $C_i$ is a left coset in $G$ with respect to the isotropy subgroup of $x_i$ in $G$. So $C_i$ is closed. 
Moreover, for an index subset $I\subseteq\{1,\ldots,m\}$ we have $g\cdot x_I=y_I$ if and only if 
$g\in\cap_{i\in I}C_i$. 

Therefore  if $g\cdot x_I=y_I$ holds for all $I$ with $|I|\leq\kappa(G)$, then by definition of the Helly dimension we conclude that $g\in C_i$ for all $i=1,\ldots,m$, hence $g\cdot x=y$. 

To see the reverse inequality, by definition of $\kappa(G)=:l$, there exist closed cosets 
$C_1,\ldots,C_l$ in $G$ such that any $l-1$ of them have a non-empty intersection, but 
$\cap_{i=1}^lC_i=\emptyset$. We have $C_i=g_iH_i$, where $H_i$ is a closed subgroup of $G$ 
and $g_i\in G$. Take for $X_i$ the $G$-homogeneous space $G/H_i$, and consider in 
$X=\prod_{i=1}^l X_i$ the points $x:=(1,\ldots,1)$ (we write $1$ for the point in $G/H_i$ corresponding to the identity element of $G$) and $y:=(g_1,\ldots,g_l)$. Then $y\notin G\cdot x$, although 
for all $I\subsetneq\{1,\ldots,l\}$ we have $G\cdot x_I=G\cdot y_I$. 
\end{proof} 

In the rest of this section $G$ is a linear algebraic group over an algebraically closed field $k$,  $X_i$ ($i=1,\ldots,m$) are irreducible 
algebraic varieties, on which $G$ acts morphically. 
We write $k(X)$ for the field of rational functions on $X$, and $k(X)^G$ for the subfield of $G$-invariants. 
The morphism $\pi_I:X\to X_I$ induces an embedding of $k(X_I)$ into $k(X)$; we shall identify $k(X_I)$ with the corresponding subfield of $k(X)$. 

\begin{theorem} \label{thm:rational} 
Suppose that $k$ has characteristic zero. 
The field $k(X)^G$ of rational invariants is generated by its subfields 
$k(X_I)^G$, where $I$ ranges over the subsets of $\{1,\ldots,m\}$ with  cardinality 
$|I|\leq \kappa(G)$. 
\end{theorem} 

We shall use the following terminology in connection with a $G$-variety $Y$. 
We say that $S\subset k(Y)^G$  {\it separates orbits in the subset} $U\subseteq Y$  
if  for any $x,y\in U$ with $G\cdot x\neq G\cdot y$ there exists an $f\in S$ defined both at $x$ and $y$, such that  $f(x)\neq f(y)$. 
We say that $S\subset k(Y)^G$  {\it separates orbits in general position} if there is a dense open subset 
$U$ of $Y$ such that $S$ separates orbits in $U$. 
Note that $S$ separates orbits in general position if and only if $k(Y)^G$ is a purely inseparable extension of $k(S)$, see for example \cite{borel}. In particular, when ${\mathrm{char}}(k)=0$, then 
$S$ separates orbits in general position if and only if $S$ generates the field $k(Y)^G$ over $k$, 
see for example Lemma 2.1 in \cite{popov-vinberg}. 

\bigskip
\begin{proofofthm}{}\ref{thm:rational}. For all $I\subseteq \{1,\ldots,m\}$ with $|I|= \kappa(G)$, 
$k(X_I)^G$ contains a finite subset $R_I$ separating orbits in general position in $X_I$
by a theorem of Rosenlicht \cite{rosenlicht}. 
Therefore there is a non-empty $G$-stable open subset $U_I$ in $X_I$ on which all elements of $R_I$ are defined, and $R_I$ separates orbits in $U_I$. Then $U:=\cap_{|I|=\kappa(G)}\pi_I^{-1}(U_I)$ is a non-empty open $G$-stable subset of $X$ on which all elements of 
$R:=\cup_I R_I\subset k(X)^G$ are defined. Moreover, we claim that $R$ separates orbits in $U$. 
Indeed, take $x,y\in U$ such that $f(x)=f(y)$ for all $f\in R$. Then for all $I$ with 
$|I|=\kappa(G)$, we have that $x_I:=\pi_I(x)\in U_I$, 
$y_I:=\pi_I(y)\in U_I$, 
and $f(x_I)=f(y_I)$ for all $f\in R_I$. 
It follows that $G\cdot x_I=G\cdot y_I$ holds for all $I$ with $|I|=\kappa(G)$. 
By Lemma~\ref{lemma:hellydimnatural} we conclude $G\cdot x=G\cdot y$. 
This shows that $R$ separates orbits in $U$. 
Since ${\mathrm{char}}(k)=0$, this implies that $R$ generates $k(X)^G$. 
\end{proofofthm} 

We saw before that in positive characteristic the Helly dimension of an algebraic group is not finite in general. Developing further the example given in the proof of  Proposition~\ref{prop:hellyinfinite} we show that the analogue of Theorem~\ref{thm:rational} also fails in positive characteristic. To be more precise, for a linear algebraic group $G$, and a product of $G$-varieties 
$X:=\prod_{i=1}^mX_i$, denote by $\rho(G,X)$ the minimal natural number 
$r$ such that the union of the subfields $k(X_I)^G$ with $|I|\leq r$ separates orbits in general position in $X$. 
Define $\rho(G)$ as the supremum of $\rho(G,X)$, as $X$ ranges over all 
products of $G$-varieties. 

\begin{proposition}\label{prop:rhokadd}  
If ${\mathrm{char}}(k)=p>0$, then for the additive group $k_{add}$ of the 
(algebraically closed) base field we have $\rho(k_{add})=\infty$. 
\end{proposition} 

\begin{proof} 
Set $G:=k_{add}$, and for $d\in\mn$, take an elementary abelian $p$-subgroup 
$P_d<G$ of rank $d$   
considered in the proof of Proposition~\ref{prop:hellyinfinite}. Take cosets 
\[g_iH_i \quad(i\in D:=\{1,\ldots,d+1\})\]
with  
$\bigcap_{i\in D}g_iH_i=\emptyset$ and 
$\bigcap_{i\in D\setminus\{j\}}g_iH_i\neq \emptyset$  for all $j\in D$. 
Consider the homogeneous spaces 
\[X_j:=G/H_j, \quad j\in D.\] 
Note that $H_j$ is the stabilizer of any point in $X_j$ (since $G$ is abelian). 
For $x\in X:=\prod_{j\in D}X_j$ set 
\[\widetilde x:=(g_1,\ldots,g_{d+1})\cdot x.\] 
By construction, $G\cdot x\neq G\cdot \widetilde x$ but we have 
\[G\cdot x_{D\setminus\{j\}}=G\cdot \widetilde x_{D\setminus\{j\}} 
\mbox{ for all }j\in D.\]  
Therefore if $I$ is a proper subset of $D$, then for any $f\in k(X_I)^G$ we have 
$f(x)=f(\widetilde x)$. 
Moreover, an arbitrary $G$-stable dense open subset $U$ of $X$ contains an 
$x\in U$ such that $\widetilde x$ is also contained in $U$. 
(Otherwise $X\subset Z\cup (g_1^{-1},\ldots,g_{d+1}^{-1})Z$ 
with $Z:=X\setminus U$, a contradiction.)  So a set of rational invariants separating orbits in general position in  $X$ necessarily contains an element that depends on all the $d+1$ 
factors $X_1,\ldots,X_{d+1}$. This construction can be performed with $d$ arbitrarily large, 
thus $\rho(G)=\infty$. 
\end{proof}

However, imposing some extra assumptions on the $G$-action on the $X_i$, one can get a bound on $\rho(G,X)$ in terms of the algebro-geometric dimension of $G$: 

\begin{proposition}\label{thm:gsd} 
Let $G$ be an algebraic group acting faithfully on all the $X_i$, $i=1,\ldots,m$. 
Then we have $\rho(G,X)\leq\dim(G)+2$. 
\end{proposition} 

\begin{corollary}\label{cor:dim(G)+2} 
Suppose ${\mathrm{char}}(k)=0$, and $G$ acts faithfully on $X_i$, $i=1,\ldots,m$. 
Then $k(X)^G$ is generated by its subfields $k(X_I)^G$ where $|I|\leq \dim(G)+2$. 
\end{corollary} 

\begin{proof} For an irreducible $G$-variety $Y$, denote by $\mds(Y)$ the minimal dimension of a stabilizer of a point in $Y$. Since the dimension of stabilizers is upper 
semicontinuous, there is a dense open subset $U$ in $Y$ such that 
$\dim(G_y)=\mds(Y)$ for all $y\in U$. It is easy to see that if $\mds(Y)>0$, and $G$ acts faithfully on $Z$, then $\mds(Y\times Z)<\mds(Y)$. Also $\mds(Y)\leq\dim(G)$. 
Returning to our situation, we may assume $m\geq \dim(G)+2$. Clearly we have 
$\mds(X_D)=\mds (X)=0$, where $D=\{1,\ldots,d\}$ and $d=\dim(G)$. 
Let $U$ be a dense open $G$-stable subset of $X_D$ such that the stabilizer $G_x$ is finite for all $x\in U$. For $x\in U$ and $i=1,\ldots,m-d$, set 
\[Z_{x,i}:=\bigcup_{g\in G_x\setminus\{1\}}\{(x,x_{d+1},\ldots,x_m)\mid gx_{d+i}=
x_{d+i}\},\]
and $Z_i:=\bigcup_{x\in U}Z_{x,i}$. Its closure $\overline{Z_i}$ is a $G$-stable closed subset of $X$. We claim that $\overline{Z_i}$ is a proper subset of $X$. 
Indeed, set $G_0:=G\setminus\{1\}$, $Y:=\prod_{j=d+1}^mX_j$. Then  
$G_0\times U\times Y$ contains the proper closed subset 
$A:=\{(g,u,x_{d+1},\ldots,x_m)\mid gu=u, \ gx_{d+i}=x_{d+i}\}$. 
Write $\pi:G_0\times U\times Y\to U\times Y$ for the projection map. 
Then $\pi(A)=Z_i$, and being the image of a morphism, this shows that $Z_i$ is constructible. Assume that $Z_i$ is dense in $X$. Then it contains a dense open subset $B$ of $X$. There exists an $u\in U$ such that $B\cap \{u\}\times Y$ is non-empty. Then $B\cap \{u\}\times Y$ is open in $\{u\}\times Y$. However, 
$Z_i\cap \{u\}\times Y$ is a proper closed subset of $\{u\}\times Y$ by construction of $Z_i$ (and the assumption that $G$ acts faithfully on $X_{d+i}$). This contradiction shows that 
$\overline{Z_i}$ is a proper closed subset of $X$, and therefore the same 
holds for $\cup_{i=1}^{m-d}\overline{Z_i}$. 
Set $W:=(U\times Y)\setminus(\cup_{i=1}^{m-d}\overline{Z_i})$. It is a $G$-stable dense open subset of $X$. It contains the dense open subset $W':=W\cap\cup_{|I|\leq d+2}\pi_I^{-1}(U_I)$, where for each 
index subset $I$ of size $\leq d+2$ we take a dense open subset $U_I$ in $X_I$ such that 
$k(X_I)^G$ separates orbits in $U_I$.  

Take $x,y\in W'$, and suppose that $f(x_I)=f(y_I)$ holds for all $I$ with $D\subseteq I$, $|I|\leq d+2$ and all $f\in k(X_I)^G$. 
Then  $Gx_D=Gy_D$, so replacing 
$y$ by an appropriate element in its orbit, we may assume that $z=x_D=y_D$.  
Similarly, for all $i=d+1,\ldots,m$, there exists $g_i\in G_z$ with $g_ix_i=y_i$. By construction of $W$, $g_i\in G_z$ is unique. Furthermore, since 
$G\cdot (z,x_i,x_j)=G\cdot (z,y_i,y_j)$ for all $i,j\in\{d+1,\ldots,m\}$, it follows that 
$g_{d+1}=g_{d+2}=\ldots=g_m=:g$, hence $gx=y$. 
This means that the union of the subfields $k(X_I)^G$ where 
$D\subset I$ and $|I|=d+2$ separates orbits in general position in $X$. 
\end{proof} 

%%%%%%%%%%%%%%%%%%%%%%%%%%%%%%%%%%%%%%%%%%%%%%%%%%%%%%%%%%%%%%

\section{Closed orbits in product varieties} \label{sec:closedorbit} 

In this section $G$ is a linear algebraic group over an algebraically closed field $k$. 
We keep the notation introduced in Section~\ref{sec:rational}, and consider products $X:=\prod_{i=1}^mX_i$ of affine $G$-varieties. 

We write $G_y$ for the stabilizer of a point $y$ in a $G$-variety $Y$. Recall that $\dim(G\cdot y)=\dim(G)-\dim(G_y)$. 
By a {\it closed orbit} we shall mean an orbit which is closed with respect to the Zariski topology of $Y$. 

\begin{definition}\label{def:delta} 
For a product $X=\prod_{i=1}^mX_i$ of affine $G$-varieties 
denote by $\delta(G,X)$ the minimal natural number $\delta$ having the following property: 
for an arbitrary $x\in X$ 
with closed $G$-orbit, 
there is a subset $I\subseteq \{1,\ldots,m\}$ with $|I|=\delta$, 
such that the orbit of 
$x_I$ is closed in $X_I$, and has the same dimension as the orbit of $x$ in $X$.  
Furthermore, define $\delta(G)$ as the supremum of 
$\delta(G,X)$, as $X$ ranges over all finite products of affine $G$-varieties. 
\end{definition} 

\begin{proposition}\label{prop:deltadef} 
Let $G$ and $X$ be as above. Suppose that for some $x\in X$ and 
$I\subseteq \{1,\ldots,m\}$, we have that $G\cdot x_I$ is closed in $X_I$, and 
$\dim(G\cdot x_I)=\dim(G\cdot x)$. Then 
for all $J\subseteq \{1,\ldots,m\}$ with  $J\supseteq I$, the orbit 
$G\cdot x_J$ is closed in $X_J$ and has the same dimension as $G\cdot x$. 
\end{proposition} 

\begin{proof} Obviously $I\subseteq J$ implies $G_{x_I}\supseteq G_{x_J}$, 
and $\dim(G\cdot x_I)\leq\dim(G\cdot x_J)$, 
so by the assumption on $\dim(G\cdot x_I)$ we get the equality 
$\dim(G_{x_I})=\dim(G_{x_J})$. Suppose that for some 
$J\supseteq I$, the orbit of $x_J$ is not closed. Take $y\in X_J$ contained 
in $\overline{G\cdot x_J}\setminus G\cdot x_J$. 
Then $\dim(G_y)>\dim(G_{x_J})$. 
Note that $y_I$ belongs to the orbit closure 
of $x_I$ (since ${\overline{Gx_I}}\times X_{J\setminus I}$ is a closed $G$-stable subset of $X_J$ containing $x_J$, hence $y$ as well).  
Since the orbit of $x_I$ is closed, $y_I$ belongs to the same orbit as $x_I$, implying $\dim(G_{y_I})=\dim(G_{x_I})$. 
Combining the above inequalities for dimensions of stabilizers  we get  
\[\dim(G_y)>\dim(G_{x_J})=\dim(G_{x_I})=\dim(G_{y_I})\geq \dim(G_y),\] 
a contradiction. This shows that 
$G\cdot x_J$ is closed for all $J\supseteq I$. 
\end{proof}

\begin{example}\label{example:non-affine} {\rm 
(i) A closed subset of a product of affine varieties may have no closed proper projection. Indeed, take $m$ dictinct complex numbers $\lambda_1,\ldots,\lambda_m$, set $X_i:=k\setminus\{\lambda_i\}$, and 
$\Delta:=\{(x,\ldots,x)\mid x\in k\}\cap \prod_{i=1}^mX_i$. 
Then $\Delta$ is closed in $X$. Take any $i\in\{1,\ldots,m\}$, and set 
$I:=\{1,\ldots,m\}\setminus\{i\}$.  
Then the complement of $\pi_I(\Delta)$ in its closure contains $(\lambda_i,\ldots,\lambda_i)$. 

(ii) Suppose that $X$ is a $G$-variety containing a dense orbit $G\cdot x$, and $X\setminus G\cdot x$ has $d$ irreducible components $Y_1,\ldots,Y_d$. Set $X_i:=X\setminus Y_i$, and consider 
$X:=\prod_{i=1}^dX_i$ with the diagonal $G$-action. Then $X$ contains the closed subset 
$\Delta:=\{(z,z,\ldots,z)\mid z\in\cap_{i=1}^d X_i\}$. Clearly, $\Delta=G\cdot \xi$, where 
$\xi=(x,\ldots,x)$. 
On the other hand, no proper projection of $\xi$ has closed orbit in $X$. Indeed, take 
$j\in\{1,\ldots,d\}$, and set $J:=\{1,\ldots,d\}\setminus\{j\}$. Then $Y_j\setminus\cup_{i\neq j}Y_i$ 
is non-emty, take an element $y$ of it. 
Then $(y,\ldots,y)$ belongs to the closure of the $G$-orbit of $\xi_J$, but $(y,\ldots,y)$ is not contained in the $G$-orbit of $\xi_J$. This construction shows that without restricting to affine varieties, one can not  expect finiteness of the number analogous to $\delta(G)$. Indeed, take for example the group 
$G:=\mc^{\times}\times\mc^{\times}$. It is easy to construct a $G$-variety $X$ with a dense orbit such that the number of irreducible components of the complement of the dense orbit in $X$ is arbitrarily large. (Start with the standard action of $G$ on $\mc^2$,  and blow up successively fixed points of $G$.)   }\end{example} 

Since any orbit of a unipotent group acting on an affine variety is closed (see for example 
\cite{borel}), $\delta(G)\leq \dim(G)$ for unipotent $G$. 

\begin{question} \label{conj:delta} 
Is $\delta(G)$ finite for any reductive algebraic group $G$? 
\end{question} 

We have an affirmative answer in certain special cases. 

\begin{proposition}\label{prop:torus} 
Let $T$ be an algebraic torus of rank $n$. Then 
$\delta(T)=2n$. 
\end{proposition} 

\begin{proof} First we show the inequality $\delta(T)\leq 2n$. 
Consider a product $X=\prod X_i$ of affine $T$-varieties. 
Without loss of generality we may assume that the 
$X_i$ (and hence $X$) are vector spaces on which $T$ acts linearly. 
We need to introduce some notation. Denote by ${\mathrm{char}}(T)$ the group of characters 
of $T$ written additively, and consider the $n$-dimensional $\mathbb{R}$-vector space ${\mathrm{char}}(T)\otimes \mathbb{R}$. For a vector $v\in V$, where $V$ is endowed with a linear action of $T$, write ${\mathrm{weight}}(v)$ for the set of 
$\lambda\in{\mathrm{char}}(T)$ such that the component of $v$ in the $\lambda$-weight space 
is non-zero, and denote by $\mathrm{supp}(v)$ the convex hull in 
${\mathrm{char}}(T)\otimes \mathbb{R}$  of the set ${\mathrm{weight}}(v)$. 
The dimension of the $T$-orbit of $v$ equals to the dimension of the subspace in 
${\mathrm{char}}(T)\otimes \mathbb{R}$ spanned by ${\mathrm{weight}}(v)$. 
It is known that the $T$-orbit of $v$ is closed if and only if $0$ is an interior point of ${\mathrm{supp}}(v)$ ({\it interior} is understood with respect to the topology of the affine subspace spanned by ${\mathrm{supp}}(v)$), see for example 
Proposition 6.15 in \cite{popov-vinberg}. 
Now assume that the orbit of $x$ is closed in 
$X$. Then $0$ is an interior point of ${\mathrm{supp}}(x)$, 
which is the convex hull of $\cup_{i=1}^m{\mathrm{weight}}(x_i)$. 
By a Carath\'eodory type theorem due to Steinitz 
(see Theorem 10.3 in \cite{steinitz}), there is a subset $A$ of at most $2n$ elements in the latter weight set whose convex hull contains $0$ as an interior point, such that $A$  spans the same affine subspace as ${\mathrm{supp}}(x)$ 
(and necessarily $A$ spans the same linear subspace as ${\mathrm{weight}}(x)$). 
Therefore there is a subset $I\subseteq\{1,\ldots,m\}$ with $|I|\leq 2n$ 
such that $A\subseteq {\mathrm{weight}}(x_I)$. 
Then ${\mathrm{weight}}(x_I)$ spans ${\mathrm{weight}}(x)$, implying that 
$\dim(T\cdot x_I)=\dim(T\cdot x)$, and zero is an interior point of 
${\mathrm{supp}}(x_I)$, implying that $T\cdot x_I$ is closed. 

Finally, to show the inequality $\delta(T)\geq 2n$ consider the 
action of $T=(k^{\times})^n$ on $k^{2n}$ given by 
\[(z_1,\ldots,z_n)\cdot (x_1,y_1,\ldots,x_n,y_n)=
(z_1x_1,z_1^{-1}y_1,\ldots,z_nx_n,z_n^{-1}y_n).\] 
Then the $T$-orbit of $x\in k^{2n}$ is closed and $n$-dimensional provided that all coordinates of 
$x$ are non-zero, but for any proper subset $I\subset\{1,\ldots,2n\}$, 
the orbit $T\cdot x_I$ is not closed or has smaller dimension than $n$. 
\end{proof} 

The number $\delta(G)$ is finite also for the complex general linear group $\gltwo$ 
of order two (and consequently for the complex special linear group 
$\sltwo$ of order two). To prove this we shall use the following generalization of the Hilbert-Mumford Criterion due to Birkes-Richardson 
(see \cite{birkes} or Theorem 6.9 in \cite{popov-vinberg}): 
Let $G$ be a complex reductive group acting on the affine variety $X$, and 
assume that the $G$-orbit of $x\in X$ is closed. 
Then $x$ belongs to the Zariski closure of the $G$-orbit of $y$ 
if and only if there exists a {\it one-parameter subgroup} 
(shortly $\onepsg$) $\rho:\mc^{\times}\to G$ 
(a homomorphism of algebraic groups) such that $\limzr y$ exists and belongs to the $G$-orbit of $x$. 

The defining representation of $\sltwo$ on $\mc^2$ induces a representation of $\sltwo$ on the space  
$\pol(\mc^2)$ of polynomial functions on $\mc^2$ in the standard way.  The degree $d$ homogeneous component $\pol_d(\mc^2)$ (called the space of  homogeneous binary forms of degree $d$) is $\sltwo$-invariant, and the spaces $\pol_d(\mc^2)$, $d=0,1,2,\ldots$, form a complete list of representatives of the isomorphism classes of irreducible polynomial $\sltwo$-modules. 
For a non-zero linear form $L\in\pol_1(\mc^2)$ and a non-zero element 
$f\in \pol_d(\mc^2)$ denote by $m_L(f)$ the maximal non-negative integer $n$ such that $L^n$ divides the binary form $f$. 

Any non-trivial $\onepsg$ in $\sltwo$ is of the following form: 
there is an element $g\in\sltwo$ and a positive integer $n$ such that 
\begin{equation}\label{eq:1psg}
\rho(z)=g\left(\begin{array}{cc}z^{-n} & 0   \\0 & z^n 
\end{array}\right)
g^{-1}.
\end{equation}
Note that if $\rho$ is of the form 
(\ref{eq:1psg}), then it 
stabilizes exactly two lines in 
$\pol_1(\mc^2)$, namely $\mc gx$ and $\mc gy$ (where $x,y\in\pol_d(\mc^2)$ stand for the usual coordinate functions on $\mc^2$). 
We shall call the non-zero scalar multiples of $gx$ or $gy$ the {\it eigenvectors} of $\rho$, and call the non-zero scalar multiples of $gx$ the 
{\it positive eigenvectors} of $\rho$ 
(note that the assumption $n>0$ in (\ref{eq:1psg2}) distinguishes between $x$ and $y$, 
and $\rho(z)\cdot gx=z^ngx$). 

\begin{theorem}\label{thm:sl2} 
We have the equality $\delta(\sltwo)=3$. 
\end{theorem} 

\begin{proof} Let $X_i$ be affine $\sltwo$-varieties, $i=1,\ldots,m$, where $m>3$, and assume that the orbit of $f\in X=\prod_{i=1}^mX_i$ is closed. 
We claim that there is a subset $I\subseteq \{1,\ldots,m\}$ with $|I|\leq 3$ such that the orbit of $f_I$ is closed and has the same dimension as the orbit of $f$. 
Since $X_i$ can be equivariantly embedded into a vector space on which 
$\sltwo$ acts linearly, we may assume at the outset that the $X_i$ are vector spaces with a linear action of $\sltwo$. 
Moreover, by Proposition~\ref{prop:deltadef} it is sufficient to prove our statement in the case when each $X_i$ is an irreducible non-trivial $\sltwo$-module, and the components $f_i$ are all non-zero. 
So $X_i=\pol_{d_i}(\mc^2)$ is the space of homogeneous binary forms of degree 
$d_i>0$, and $f_i\in\pol_{d_i}(\mc^2)$. The following terminology will be used both in the present proof and in the proof of Theorem~\ref{thm:gl2}: 

\begin{definition} \label{def:comroothighmult} 
We say that the non-zero linear form $L$ is a common root of high multiplicity of some  binary forms $b_i\in\pol_{d_i}(\mc^2)$, $i$ ranging over some index set, if $m_L(b_i)\geq d_i/2$ for all $i$.  
\end{definition} 

\noindent Note that $\lim_{z\to 0}\rho(z)f$ exists for some non-trivial $\onepsg$ $\rho$ 
if and only if $L$ is a common root of high multiplicity of the $f_i$ ($i=1,\ldots,m$), where $L$ is the positive eigenvector of $\rho$. Furthermore, 
if $f_j$ ($j\in J$)  have no common root with high multiplicity for some index subset 
$J\subseteq \{1,\ldots,m\}$, then the orbit of $f_J$ is closed in $X_J$ by the criterion in \cite{birkes} mentioned above. 

Assume that the components of $f$ have a common root $L$ of high multiplicity.  
Let $\rho$ be a one-parameter subgroup of $\sltwo$,  such that $L$ is the  positive eigenvector of $\rho$. 
Then $\lim_{z\to 0}\rho(z)\cdot f_i=h_i$ exists and equals zero or 
$h_i=c_i(LL')^{e_i}$, where $c_i\in\mc$, $L'$ is a linear form independent of $L$, and $d_i=2e_i$. By assumption $h:=\lim_{z\to 0}\rho(z)\cdot f$ belongs to the orbit of $f$. 
It follows that there are two non-proportional linear forms $L_1,L_2$ such that $f_i=c_i(L_1L_2)^{e_i}$, with $0\neq c_i\in\mc$. 
Clearly $I=\{1\}$ satisfies the requirements in this case. 

Now assume that the components of $f$ have no common root of high multiplicity.  Note that a binary form of degree $d$ has at most one root of multiplicity 
$\geq d/2$, unless $d$ is even, and the binary form is a power of the product of 
two linear forms (when it has exactly two roots up to non-zero scalar multiple) of multiplicity $d/2$. 

If there is a component, say $f_1$, that has no root of high multiplicity, then 
the orbit of $f_J$  is closed for all $J$ containing $1$. 
Since $\dim(\sltwo)=3$, one can select two components, say $f_2,f_3$ such that $\dim(\sltwo_{(f_1,f_2,f_3)})=\dim(\sltwo_f)$. 
Then $I=\{1,2,3\}$ has the desired properties. 

If all  components have a root with high multiplicity, and 
there are two components, say $f_1$ and $f_2$ that have no common root of high multiplicity, then $\dim(\sltwo_{(f_1,f_2)})\leq 1$, and selecting a third component if necessary, say $f_3$, we may achieve that 
$\dim(\sltwo_{(f_1,f_2,f_3)})=\dim(\sltwo_f)$. 
Then $I=\{1,2,3\}$ satisfies the requirements. 

The only missing case is that $d_i=2e_i$ is even for all $i=1,\ldots,m$, 
and there are linear forms $L_1,L_2,L_3$, pairwise non-proportional, such that each component of $f$ is a non-zero scalar multiple of a power of $L_1L_2$ or 
$L_1L_3$ or $L_2L_3$ (all occuring). Without loss of generality, we may assume 
that $f_1=(L_1L_2)^{e_1}$, $F_2=(L_1L_3)^{e_2}$, and $f_3=(L_2L_3)^{e_3}$. 
Then for $I=\{1,2,3\}$, we have $\sltwo_{f_I}$ is finite, and the orbit of $f_I$ is closed. 

Finally, to show the reverse inequality $\delta(\sltwo)\geq 3$, consider 
$X_1\times X_2\times X_3$, where $X_i\cong\pol_2(\mc^2)$ for $i=1,2,3$, and 
the point $f=(xy,x(x+y),y(x+y))$. Then none of the orbits of $(f_1,f_2)$, 
$(f_1,f_3)$, $(f_2,f_3)$ are closed. (The orbits of $f_1$, $f_2$, and $f_3$ are all closed, but their dimension is less than the dimension of the orbit of $f$.)  
\end{proof} 

Recall that an irreducible rational $\gltwo$-module is isomorphic to 
$\pol_d(\mc^2)\otimes \mc_{\det^e}$ for some $d\in\mn_0$ and $e\in\mz$, where $\mc_{\det^e}$ denotes the one-dimensional vector space $\mc$ on which $g\in\gltwo$ acts by multiplication by the scalar $\det(g)^e$. We shall identify the underlying vector space of $\pol_d(\mc^2)\otimes \mc_{\det^e}$ with 
$\pol_d(\mc^2)$ in the obvious way and call its elements binary forms, 
and keep the terminology and notation introduced for binary forms in the case of $\sltwo$.  

Up to conjugacy in $\gltwo$, a $\onepsg$ $\rho$ in $\gltwo$ is of the form 
\begin{equation}\label{eq:1psg2}
\rho(z)=\left(\begin{array}{cc}z^{r-n} & 0   \\0 & z^{r+n} 
\end{array}\right), 
\end{equation}
where $2r\in\mz$, $2n\in\mn_0$, and $r-n\in\mz$. 
In fact, for any $\onepsg$ $\tau$ in $\gltwo$ there is an element $g\in\sltwo$ such that 
$g^{-1}\tau(z)g$ is of the form (\ref{eq:1psg2}), and  
we shall use the notation 
$r(\tau):=r$ and $n(\tau):=n$. 
When $n(\tau)\neq 0$, the $\onepsg$ $\tau$ stabilizes exactly two lines in 
$\pol_1(\mc^2)$, namely $\mc gx$ and $\mc gy$. We shall call the non-zero scalar multiples of $gx$ or $gy$ the {\it eigenvectors} of $\tau$, and call the non-zero scalar multiples of $gx$ the 
{\it positive eigenvectors} of $\tau$ 
(note that the assumption $n>0$ in (\ref{eq:1psg2}) distinguishes between $x$ and $y$). 

\begin{theorem}\label{thm:gl2} 
We have the equality $\delta(\gltwo)=5$. 
\end{theorem} 

\begin{proof} Write $G$ for $\gltwo$. Similarly to the proof of Theorem~\ref{thm:sl2}, it is sufficient to show that if $V_i$ 
($i\in M$) is a finite set of irreducible rational $\gltwo$-modules, 
and the $G$-orbit of $v\in V:=\bigoplus_{i\in M}V_i$ is closed, 
then there is a subset $I\subseteq M$ with $|I|\leq 5$ such that 
the orbit of $v_I$ in $V_I=\bigoplus_{i\in I}V_i$ is closed, and $\dim(G\cdot v_I)=\dim(G\cdot v)$. Clearly it is sufficient to deal with the case when  $v_i\neq 0$ for all $i\in M$, and that none of the $V_i$ is the trivial module, so we shall assume that this is the case. 
The irreducible $\gltwo$-module $V_i$ is isomorphic to 
$\pol_{d_i}(\mc^2)\otimes\mc_{\det^{e_i}}$. 
There is a decomposition $V=V_+\oplus V_0\oplus V_-$, where 
\[V_+=\bigoplus_{2e_i-d_i>0}V_i,\quad V_0=\bigoplus_{2e_i-d_i=0}V_i,\quad 
V_-=\bigoplus_{2e_i-d_i<0}V_i.\] 
For $v\in V$ we write 
$v=v_++v_0+v_-$, where $v_+\in V_+$, $v_0\in V_0$, $v_-\in V_-$. 
For a subset $J\subseteq M$, set 
$J_+:=\{j\in J\mid 2e_j-d_j>0\}$, 
$J_-:=\{j\in J\mid 2e_j-d_j<0\}$, and 
$J_0:=\{j\in J\mid 2e_j-d_j=0\}$. 
We have $\gltwo=\mc^{\times}\cdot \sltwo$, where $\mc^{\times}$ is identified with the subgroup of scalar matrices. Since scalar matrices act on $V_0$ trivially, the $\gltwo$-orbits and $\sltwo$-orbits on $V_0$ coincide. Thus by  Theorem~\ref{thm:sl2} it is sufficient to deal with the case when $v\notin V_0$. 
If $V_-=0$, then with $\rho(z)=
\left(\begin{array}{cc}z & 0 \\0 & z\end{array}\right)$ we have 
$\limzr \cdot v_+=0$, hence $\limzr\cdot v=v_0$. Since the orbit of $v$ is closed, it follows that $v=v_0$. Similarly, $V_+=0$ implies $v=v_0$. 
Therefore we may assume that both $V_+$ and $V_-$ are non-zero. 
This implies that if $\limzr \cdot v$ exists for some $\rho$ as in 
(\ref{eq:1psg2}), then $n\neq 0$ (and hence $n>0$). 

With $\rho(z)$ as in (\ref{eq:1psg2}) and for 
$x^iy^{d-i}\in\pol_d(\mc^2)\otimes\mc_{\det^e}$ 
we have 
\begin{equation} \label{eq:rhoxy} 
\rho(z)\cdot x^iy^{d-i}=z^{r(2e-d)+n(2i-d)}x^iy^{d-i},\end{equation}  
hence $\limzr\cdot f$ exists for some $f\in \pol_d(\mc^2)\otimes\mc_{\det^e}$ 
and $\rho$ as in (\ref{eq:1psg2}) if and only if 
\[r(2e-d)+n(2m_x(f)-d)\geq 0.\]  
For $j\in M_+\cup M_-$ and a linear form $L$ introduce the notation 
$\mu_L(v_j):=\frac{d_j-2m_L(v_j)}{2e_j-d_j}$. 
Consequently, if $J$ is a subset of $M$, and $\rho$ is a $\onepsg$ with $n(\rho)\neq 0$ and positive eigenvector $L$, then $\limzr \cdot v_J$ exists if and only if 
we have 
\begin{eqnarray}\label{eq:r/n}
\begin{cases} r(\rho)/n(\rho)\geq \mu_L(v_j)\mbox{ for all }j\in J_+
\\  r(\rho)/n(\rho)\leq \mu_L(v_j)\mbox{ for all }j\in J_- 
\\  m_L(v_j)\geq d_j/2 \mbox{ for all } j\in J_0\end{cases}  
\end{eqnarray}
It follows that if $v_J$ has a non-zero component both in $V_+$ and 
in $V_-$, then 
$\limzr\cdot v_J$ does not exist for all non-trivial $\onepsg$ if and only if for all linear forms $L\in\pol_1(\mc^2)$ we have 
\begin{eqnarray}\label{eq:maxmin}\max\{\mu_L(v_j)\mid j\in J_+\}> 
\min\{\mu_L(v_j)\mid j\in J_-\} \\ \notag
\mbox{ or }m_L(v_j)<d_j/2\mbox{ for some }j\in J_0.
\end{eqnarray} 
Let us introduce the following ad hoc terminology: if for some $J\subseteq M$ and a linear form $L$ condition (\ref{eq:maxmin}) holds, then we say that 
$J$ {\it rules out} $L$. Observe that if $J$ rules out $L$, then there exists an at most two element subset of $J$ ruling out $L$, and every index set containing $J$ rules out $L$.  
Observe also that $\max\{\mu_L(v_j)\mid j\in J_+\}$ is positive if $m_L(v_j)<d_j/2$ for some $j\in J_+$, and $\min\{\mu_L(v_j)\mid j\in J_-\}$ is negative if  $m_L(v_j)<d_j/2$ for some $j\in J_-$. Consequently, $J$ rules out $L$ unless $L$ is a common root with high multiplicity of the $v_j$ with $j\in J_+\cup J_0$, or $L$ is a common root of high multiplicity of the $v_j$ with $j\in J_-\cup J_0$ 
(see Definition~\ref{def:comroothighmult}). 

{\it Case 1: Assume that for all non-trivial $\onepsg$ $\rho$, 
$\limzr\cdot v$ does not exist.} That is, $M$ rules out all linear forms $L$.  
We shall construct a subset $J\subseteq M$ of size $\leq 5-\dim(G_{v_J})$ ruling out all $L$, and with both $J_+$ and $J_-$ non-empty. Then clearly there exists an index subset $I\supseteq J$ of size $\leq 5$, such that 
$\dim(G_{v_I})=\dim(G_v)$. Moreover, $I$ obviously rules out all linear forms $L$, hence by the Birkes-Richardson criterion cited above, the orbit of $v_I$ is closed.  
The rough idea to carry out this plan is the following: a binary form of positive degree has at most two roots with high multiplicity (from now on we identify a linear form and its non-zero scalar multiples). 
So if $v_1\in V_+$ and $v_2\in V_-$ have positive degrees, then 
$\{1,2\}$ rules out all linear forms with the possible exception of at most four linear forms. So adding at most $4\cdot 2=8$ indices to $\{1,2\}$ we get a set $J$ ruling out all linear forms. Adding at most $\dim(G)-1$ more indices we get 
an $I$ satisfying also the desired condition on the dimension of $G\cdot v_I$. 
This shows the existence of a set $I$ of size $\leq 13$ satisfying all the other requirements. To get the desired $I$ with size $\leq 5$ needs a bit lengthy 
case-by-case analysis, and a more careful look at stabilizers. 

For a non-zero $f\in \pol_d(\mc^2)\otimes \mc_{\det^e}$ where $2e-d\neq 0$, one of the following options holds: 
\begin{itemize}
\item[(a)] $f$ has no root with high multiplicity (this forces $d\geq 3$); 
then $G_f$ is finite. 
\item[(b)] $f$ has exactly one root with high multiplicity (this forces 
$d\geq 1$); then $\dim(G_f)\leq 1$, unless $f=L^d$ for some linear form $L$, when $\dim(G_f)=2$. 
\item[(c)] $f$ has exactly two roots $L_1,L_2$ with high multiplicity (i.e. 
$f$ is a non-zero scalar multiple of $(L_1L_2)^h$, where $h\in\mn$ and $d=2h$); then $\dim(G_f)=1$. 
\item[(d)] $d$ (the degree of $f$) equals zero (hence any non-zero linear form $L$ is a root of $L$ with high multiplicity); then $(G_f)^{\circ}=\sltwo$. 
\end{itemize} 

{\it Case 1.1: There is a component $v_1$ of $v_+$ and $v_2$ of $v_-$ having no root of high multiplicity.} 
By the above considerations, $I=\{1,2\}$ obviously has the desired properties. 

{\it Case 1.2: There is a component $v_1$ of $v_+$ that has no root of high multiplicity, and there is a component $v_2$ of $v_-$ that has exactly one root $L$ of high multiplicity.} 
Then $\{1,2\}$ rules out all linear forms except possibly $L$. 
There are at most two indices, say $a,b$, such that $\{a,b\}$ rules out $L$. 
Then $I=\{1,2\}\cup\{a,b\}$ rules out all linear forms and $v_I$ has finite stabilizer. 

{\it Case 1.2': Interchange $+$ and $-$ in Case 1.2.} 

From now on we shall automatically assume in each case that we are not in any of the cases 
considered before, or in their pair obtained by interchanging $+$ and $-$. 

{\it Case 1.3: There is a component $v_1$ of $v_+$ that has no root with high multiplicity, and there is a component $v_2$ of $v_-$ having exactly two roots 
$L_1$ and $L_2$ with high multiplicity.} 
Then $\mu_{L}(v_1)>0$ and $\mu_{L}(v_2)\leq 0$ for all linear forms $L$, hence 
$I=\{1,2\}$ works.

{\it Case 1.4: There is a component $v_1$ of $v_+$ that has no root with high multiplicity, and there is a component $v_2$ of degree $0$.} 
Then $\mu_{L}(v_1)>0$ and $\mu_{L}(v_2)=0$ for all linear forms $L$, hence 
$I=\{1,2\}$ works.

{\it Case 1.5 (i): $v_1\in V_+$ has exactly one root $L$ with high multiplicity, and 
$L$ is the only root of $v_2\in V_-$ with high multiplicity.} 
Then $\{1,2\}$ rules out all linear forms except possibly $L$. There exist 
$\{a,b\}$ ruling out $L$. Then $J=\{1,2\}\cup\{a,b\}$ rules out all linear forms.  
Moreover, the stabilizer of $v_J$ has dimension $\leq 1$. 
Indeed, $\dim(G_{v_{\{1,2\}}})\leq 1$ unless $v_1$ and $v_2$ are powers of $L$, when 
$\dim(G_{v_{\{1,2\}}})=2$. Note that one of $v_a$, $v_b$ is not a power of $L$, hence 
$\dim(G_{v_J})<\dim(G_{v_{\{1,2\}}})$. 

{\it Case 1.5 (ii): $v_1\in V_+$ has exactly one root $L_1$ with high multiplicity, and $v_2\in V_-$ has exactly one root $L_2$ with high multiplicity, where $L_1$ and $L_2$ are different (non-proportional).} 
Then $\dim(G_{(v_1,v_2)})\leq 1$, and $\{1,2\}$ rules out all linear forms except possibly $L_1$ or $L_2$. If there is a component $v_3\in V_+$  with 
$\mu_{L_1}(v_3)\geq 0$ or $v_3\in V_0$ with $m_{L_1}(v_3)<d_3/2$, then $\{3,2\}$ rules out $L_1$. 
Otherwise  we may assume that $\mu_{L_1}(v_1)<0$ is maximal among 
$\{\mu_{L_1}(v_i)\mid i\in M_+\}$, and there is a $v_3\in V_-\cup V_0$ such that 
$\{1,3\}$ rules out $L_1$. Similarly one finds a component $v_4$ such that 
$\{1,2,4\}$ rules out $L_2$. Then $J=\{1,2,3,4\}$ has the desired properties. 
 
{\it Case 1.6: $v_1\in V_+$ has exactly one root $L_1$ with high multiplicity, and $v_2\in V_-$ has exactly two roots with high multiplicity (that is,  $v_2=(L_2L_3)^h$, where $L_2,L_3$ are non-proportional linear forms and $h\in \mn$).} Then $\{1,2\}$ rules out all linear forms except possibly $L_1$ (note that $\mu_{L_2}(v_2)=\mu_{L_3}(v_2)=0$), and the stablilizer of $(v_1,v_2)$ has dimension $\leq 1$. Take at most two indices $\{a,b\}$ ruling out $L_1$. 
Then $J=\{1,2,a,b\}$ has the desired properties. 
 
{\it Case 1.7: $v_1\in V_+$ has exactly one root $L$ with high multiplicity,
and $v_2\in V_-$ has degree zero.} 
All linear forms different from $L$ are ruled out by $\{1,2\}$. 
Moreover, the stabilizer of $(v_1,v_2)$ has dimension $\leq 1$. 
Then we may take $J:=\{1,2\}\cup\{a,b\}$, where $\{a,b\}$ rules out $L$. 

{\it Case 1.8: $v_1\in V_+$ has exactly two roots $L_1,L_2$ with high multiplicity, $v_2\in V_-$ has exactly two roots $L_3,L_4$ with high multiplicity, and $\{\mc L_1,\mc L_2\}\neq \{\mc L_3,\mc L_4\}$.} 
Then  $\dim G_{(v_1,v_2)}=0$, and $\{1,2\}$ rules out  all linear forms, except when 
$\{\mc L_1,\mc L_2\}\cap\{\mc L_3,\mc L_4\}$ is non-empty, say $\mc L_2=\mc L_4$. In the latter case 
$\{1,2\}$ rules out all linear forms except $L_2$. Take $\{a,b\}$ ruling pout $L_2$. Then 
$I=\{1,2,a,b\}$  has the desired properties.  

{\it Case 1.9: $v_1\in V_+$ has exactly two roots $L_1,L_2$ with high multiplicity, and 
$v_i=c_i(L_1L_2)^{h_i}$ with $c_i\in\mc^{\times}$, $h_i\in\mn_0$ 
for all $i\in M_-$.}   Take any index, say $2\in M_-$. 
The stabilizer of $v_1$ has dimension $1$. Moreover, a linear form different from $L_1$ and $L_2$ is ruled out by $\{1,2\}$. Note that $\mu_{L_1}(v_i)=0=
\mu_{L_2}(v_i)$ for all $i\in M_-$. It follows that there exist 
indices $a_1,a_2\in M_+\cup M_0$ such that $\{a_1,2\}$ rules out $L_1$ and $\{a_2,2\}$ rules out $L_2$. So $J:=\{1,2,a_1,a_2\}$ has the desired properties.  

{\it Case 1.10: $\deg(v_i)=0$ for all $i\in M_+\cup M_-$.} 
We may assume that $1\in M_+$ and $2\in M_-$. Note that $(G_{v_1})^{\circ}=\sltwo$, hence if $1\in K\subseteq M$, then $\dim(G\cdot v_K)=
\dim(\sltwo\cdot v_K)+1$. 
Now $\sltwo$ fixes $v_+$ and $v_-$, so by the basic assumption of Case 1., the $\sltwo$-orbit of $v_0$ is closed. By Theorem~\ref{thm:sl2} we may choose at most three indices 
$a,b,c\in M_0$ such that $\sltwo\cdot (v_a,v_b,v_c)$ is closed and has 
dimension $\dim(\sltwo\cdot v_0)=\dim(\sltwo\cdot v)$. 
Set $I:=\{1,2,a,b,c\}$, then $\dim(G\cdot v_I)=\dim(G\cdot v)$. 
Suppose that $\limzr v_I$ exists for some $\onepsg$ $\rho$. 
Then $1,2\in I$ forces that $\rho$ is contained in $\sltwo$. 
Since the $\sltwo$-orbit of $v_I$ is closed, $\limzr v_I$ belongs to the orbit of $v_I$. By the Birkes-Richardson Criterion cited above, this implies that the orbit of $v_I$ is closed.

{\it Case 2:  
$\limzr\cdot v$ exists for some non-trivial $\onepsg$ $\rho$.} 
Then $\limzr\cdot v$ belongs to the orbit of $v$ and is fixed by $\rho$, hence $v$ is fixed by a non-trivial $\onepsg$ $\tau$. 
Formula (\ref{eq:rhoxy}) shows that replacing $v$ by an appropriate element in its orbit, we may assume that 
$v_i=c_ix^{a_i}y^{d_i-a_i}$ for $i=1,\ldots,m$, (where $c_i\in\mc^{\times}$),  
$\mu_x(v_i)=r(\tau)/n(\tau)=:q$ for all $i\in M_+\cup M_-$, and $m_x(v_j)=d_j/2$ for all $j\in M_0$. 

{\it Case 2.1:  $v$ has a component $v_1$ with positive degree.}   
There is a subset $1\in J\subseteq M$ with $|J|\leq 3$, and both $J_+$, $J_-$ non-empty. 
Let $\lambda$ be a non-trivial $\onepsg$ with positive eigenvector $L$. 
If $L$ is non-proportional to $x$ or $y$, then $\mu_L(v_i)\geq 0$ if $i\in J_+$,  
$\mu_L(v_i)\leq 0$ if $i\in J_-$, $\mu_L(v_1)\neq 0$ if $v_1\notin V_0$, and 
$m_L(v_1)<d_1/2$ if $v_1\in V_0$. Consequently, $\lim_{z\to 0}\lambda(z)\cdot v_J$ does not exist. 
If $L=x$, then $\mu_L(v_i)=q$ for all $i\in J_+\cup J_-$ and $m_L(v_i)=d_i/2$ if $i\in J_0$. 
So if $r(\lambda)/n(\lambda)\neq q$, , then $\lim_{z\to 0}\lambda(z)\cdot v_J$ does not exist. If $r(\lambda)/n(\lambda)=q$, 
then $\lim_{z\to 0}\lambda(z)\cdot v =g\cdot v$, where $g\in\sltwo$ satisfies that $g^{-1}\lambda(z)g$ is of the form (\ref{eq:1psg2}). The case $L=y$ is dealt with similarly. It follows that $G\cdot v_K$ is closed for all $K\supseteq J$. 
Since $\dim(G_{v_J})\leq 2$, adding at most two indices to $J$ we obtain an $I$ with the desired properties. 

{\it Case 2.2: $d_i=0$ for all $i\in M_+\cup M_-$.} 
Choose $v_1\in V_+$ and $v_2\in V_-$. Note that if $1\in K$, then 
$(G_{v_K})^{\circ}\leq\sltwo$, 
hence $\dim (G\cdot v_K)=\dim(\sltwo\cdot v_K)+1$. 
If $\lim_{z\to 0}\lambda(z)\cdot(v_1,v_2)$ exists for some $\onepsg$ $\lambda$, then 
$\lambda$ is a $\onepsg$ in $\sltwo$, hence fixes $(v_1,v_2)$. It follows that 
if $K\supseteq \{1,2\}$, then the $\gltwo$-orbit of $v_K$ is closed if and only if the $\sltwo$-orbit of $v_{K_0}$ is closed. Hence the existence of the desired $I$ follows from Theorem~\ref{thm:sl2}. 

Finally we prove the reverse inequality $\delta(\gltwo)\geq 5$. 
Consider the direct sum $V$ of five irreducible  $\gltwo$-representations,  
where $V_+=V_1\cong\pol_0(\mc^2)\otimes\det$, 
$V_-=V_2\cong\pol_0(\mc^2)\otimes\det^{-1}$, 
$V_0=V_3\oplus V_4\oplus V_5\cong (\pol_2(\mc^2)\otimes \det^0)^{\oplus 3}$. 
Consider the point $v:=(1,1,xy,x(x+y),y(x+y))\in V$. 
The $\sltwo$-orbit of $v_0$ is closed (see the proof of Theorem~\ref{thm:sl2}), hence the $G$-orbit of $v$ is closed (see Case 2.2). 
The closure of the orbit of $(v_-,v_0)$ contains $(0,v_0)$, and similarly, the 
closure of the orbit of $(v_+,v_0)$ contains $(0,v_0)$. 
The closure of $(1,1,xy,x(x+y))$ contains $(1,1,xy,xy)$. Similarly, neither of the orbits of $(1,1,xy,y(x+y))$ or $(1,1,x(x+y),y(x+y))$ is closed. 
(The orbit of $v_0$ is closed, but has smaller dimension than the orbit of $v$.) 
This example implies that desired inequality.   
\end{proof} 

For an affine $G$-variety $X$, denote by $k[X]$ the coordinate ring of $X$, and 
$k[X]^G$ the subalgebra of $G$-invariants. Following \cite{derksen-kemper} we say that a subset $S\subset k[X]^G$ is a {\it separating system of polynomial invariants} if whenever $x,y\in X$ can be separated by polynomial invariants 
(i.e. $f(x)\neq f(y)$ for some $f\in k[X]^G$), then $x$ and $y$ is separated by an element of $S$ (i.e. $h(x)\neq h(y)$ for an appropriate $h\in S$). There is a recent interest in separating systems of polynomial invariants, see for example \cite{kemper}, \cite{draisma-kemper-wehlau}, \cite{d:sep}, 
\cite{grosshans:2007}, \cite{dufresne}, \cite{neusel-sezer}, \cite{dufresne-elmer-kohls}, and the original motivation for the present paper came also from this topic. 

\begin{definition}\label{def:sigma} 
Given affine $G$-varieties $X_1,\ldots,X_m$, 
let $\sigma(G,X_1,\ldots,X_m)$ denote the minimal natural number $d$ that 
polynomial invariants depending only on $d$ factors form a separating system in $k[X_1\times\cdots\times X_m]^G$ 
(here we say that $f$ depends only on $d$ factors if there is an index subset 
$I\subseteq \{1,\ldots,m\}$ of size $|I|=d$, such that 
$f\in \pi_I^*(k[X_I]^G)$, where $\pi_I^*$ is the comorphism of the projection 
$X\to X_I$). 
Moreover, write $\sigma(G)$ for the supremum of $\sigma(G,X_1,\ldots,X_m)$ over all finite tuples of affine $G$-varieties. 
\end{definition} 

Our interest in the numbers $\delta(G)$ and $\kappa(G)$ stems from Lemma~\ref{lemma:kappa+delta} below.  

\begin{lemma}\label{lemma:kappa+delta} 
For a reductive algebraic group $G$ we have the inequality 
\[\sigma(G)\leq \delta(G)+\kappa(G).\]  
\end{lemma} 

\begin{proof} Assume $\delta=\delta(G)$ is finite (otherwise the statement is vacuous), and let $X_1,\ldots,X_m$ be affine $G$-varieties with $m>\delta+\kappa(G)$. 
Recall that for an action of a reductive group on an affine variety, two points can be separated by polynomial invariants if and only if their orbit closures do not intersect, moreover, the closure of any orbit contains a unique closed orbit, see for example \cite{grosshans:1997}. 
Therefore it is sufficient to show that if  $x=(x_1,\ldots,x_m)$ and 
$y=(y_1,\ldots,y_m)$ have closed orbits, and none of their projections into subfactors 
with $d=\delta(G)+\kappa(G)$ components can be separated by polynomial invariants, then  
$x$ and $y$ belong to the same orbit. By symmetry we may assume that 
$\dim(Gx)\geq\dim(Gy)$. 
By the definition of $\delta(G)$ we have that after a possible renumbering of the components $X_i$, the orbit of $x_D$ is closed in $X_D$, and 
has the same dimension as $Gx$, where 
$D=\{1,\ldots,\delta(G)\}$. By Proposition~\ref{prop:deltadef}, 
the orbit of $x_J$ is closed for all $J\supseteq D$. 
This means that for such a $J\supseteq D$, the orbit closures of $x_J$ and $y_J$ 
intersect non-trivially if and only if $Gx_J=Gy_J$ (recall that 
$\dim(Gx_J)=\dim(Gx)\geq\dim(Gy)\geq\dim(Gy_J)$). 
It follows from our assumption that $Gx_J=Gy_J$ for all $J\supseteq D$, with 
$|J|\leq\delta(G)+\kappa(G)$ (otherwise $x_J$ and $y_J$ can be separated by polynomial invariants). 
In particular, $Gx_D=Gy_D$, so replacing $y$ by an appropriate elements in its orbit, we may assume that $x_D=y_D$. Moreover, denoting by $H$ the stabilizer of 
$x_D$, for all $J\subseteq \{\delta(G)+1,\ldots,m\}$ with $|J|\leq\kappa(G)$, 
we have $Hx_J=Hy_J$. Since $\kappa(H)\leq\kappa(G)$, by definition of the Helly dimension $(x_{\delta+1},\ldots,x_m)$ and $(y_{\delta+1},\ldots,y_m)$ belong to the same 
$H$-orbit, implying that $x$ and the original $y$ belong to the same $G$-orbit.  
\end{proof} 

An affirmative answer to Question~\ref{conj:delta} would imply 
an affirmative answer to the following: 

\begin{question}\label{conj:sigma} 
In characteristic zero, is $\sigma(G)$ finite for any 
reductive group $G$?  
\end{question} 

%%%%%%%%%%%%%%%%%%%%%%%%%%%%%%%%%%%%%%%%%%%%%%%%%%%%%%%%%%%%%%

\section{Separating systems}\label{sec:separating} 

In \cite{d:sep} the number of variables in a separating system is bounded in terms of the dimension of the representation. Using an idea from \cite{grosshans:2007}, we sharpen Theorem 2.2 of \cite{d:sep}. 

In this section $G$ will be a linear algebraic group over an algebraically closed field $k$ of arbitrary characteristic, and $V,W$ will be finite dimensional $G$-modules (i.e. $G$ acts linearly and morphically on them). Write $V^m$ for the $m$-fold direct sum of $V$. 
Given a subset $I=\{i_1,\ldots,i_d\}\subseteq\{1,\ldots,m\}$, write 
$\pi_I:V^m\oplus W\to V^d\oplus W$ for the projection 
$(v_1,\ldots,v_m,x)\mapsto (v_{i_1},\ldots,v_{i_d},x)$. 

\begin{theorem}\label{thm:separating} 
Take two points $y,z\in V^m\oplus W$, and assume that 
there is a polynomial invariant $f\in k[V^m\oplus W]^G$ 
such that $f(y)\neq f(z)$. Then there exists a subset 
$I\subseteq\{1,\ldots,m\}$ with $|I|=d\leq 1+\dim_{k}(V)$, 
and a polynomial invariant $c\in k[V^d\oplus W]^G$ with 
$c(\pi_I(y))\neq c(\pi_I(z))$. 
\end{theorem} 

\begin{proof} 
In the special case when $G$ is reductive the result 
was proved in \cite{d:sep} (cf. Theorem 3.2). 
Moreover, the proof works verbatim when 
$W$ is an arbitrary affine $G$-variety.  

Now we apply the "transfer principle" to extend the result from reductive groups to arbitrary groups, similarly to Grosshans \cite{grosshans:2007}, who extended this way his "p-root closure theorem" for arbitrary groups from the special case of reductive groups. The action of $G$ on $V\oplus W$ gives a homomorphism of $G$ into $H:={\mathrm{GL}}(V)\times{\mathrm{GL}}(W)$. 
Write $G_1$ for the image of $G$ in $H$ under this homomorphism. 
Then $V,W$ are naturally $G_1$-modules, and it is obviously sufficient to prove our statement for the case when $G=G_1$. 
Thus from now on we assume that $G\subseteq H$. 
Note that the group $H$ is reductive. 
Set $U:=V^m\oplus W$.  
Consider the $G\times H$-variety $U\times H$, with the action 
$(g,h)\cdot (u,h')=(hu,hh'g^{-1})$. 
The comorphism $\Phi^*$ of the $H$-invariant morphism 
$\Phi:U\times H\to U$, $(u,h)\mapsto h^{-1}u$ 
identifies $k[U]^G$ with $k[U\times (H/G)]^H$ (this is called the transfer principle or adjunction argument, cf. \cite{grosshans:1997}).  

Take $y,z\in U$ and $f\in k[U]^G$ with $f(y)\neq f(z)$. 
Keep the notation $1_H$ for the image of the identity element of $H$ under the quotient morphism $H\to H/G$.  Then we have $\Phi^*(f)(y,1_H)\neq \Phi^*(f)(z,1_H)$, so $(y,1_H)$ and $(z,1_H)$ can be separated by an 
$H$-invariant on $U\times (H/G)$. 
Since $H$ acts rationally on $k[H]$ by left translation, and 
$k[H/G]$ is an $H$-stable subalgebra, it follows that $H$ acts 
rationally on $k[H/G]$. Hence there is a finitely generated 
$H$-stable subalgebra $A$ in $k[H/G]$ such that 
$\Phi^*(f)\in k[U]\otimes A$. 
Write $Z$ for the affine $H$-variety ${\mathrm{Spec}}(A)$, and 
$\alpha:H/G\to Z$ for the morphism whose comorphism is the 
inclusion $A\to k[H/G]$. 
Then $\Phi^*(f)$ belongs to $k[U\times Z]$, and 
$\Phi^*(f)(y,\alpha(1_H))\neq \Phi^*(f)(z,\alpha(1_H))$. 
By Theorem 3.2 in \cite{d:sep} (and the remark in the first paragraph of the present proof) there is an invariant $b\in k[V^d\times (W\times Z)]^H$  
with $d\leq 1+\dim_{k}(V)$, and an $I\subseteq \{1,\ldots,m\}$ with 
$|I|=d$, such that $b(\pi_I(y),\alpha(1_H))\neq b(\pi_I(z),\alpha(1_H))$. 
Identifying $b$ with its image under 
${\mathrm{id}}\otimes\alpha^*:k[(V^d\times W)\times Z]\to 
k[(V^d\times W)\times (H/G)]$ 
we find that $b\in k[V^d\times W\times (H/G)]^H$ has the property that 
$b(\pi_I(y),1_H)\neq b(\pi_I(z),1_H)$.  
Define $c\in k[V^d\times W]^G$ by $c(x)=b(x,1_H)$. 
Then we have $c(\pi_I(y))\neq c(\pi_I(z))$. 
\end{proof} 
 
\begin{remark}\label{rem:grosshans} {\rm 
The analogues of Theorem~\ref{thm:separating} with the weaker bound 
$d\leq 2\dim_{k}(V)$ was the starting point of \cite{d:sep}. 
An example is given in loc.cit. showing that the bound 
$d\leq 1+\dim_{k}(V)$ is sharp. }
\end{remark}

%%%%%%%%%%%%%%%%%%%%%%%%%%%%%%%%%%%%%%%%%%%%%%%%%%%%%%%%%%

\section{Lie groups}\label{sec:lie} 

The following example shows that it is essential in Definition~\ref{def:hellydim} to consider finite systems of cosets only. 

\begin{example} \label{example:infinitecosets} 
{\rm Consider the additive group $\mz$ of integers as a closed subgroup 
of the one dimensional real Lie group $(\mr,+)$. 
Now for all $n\in\mn$ define the coset 
\[C_n:=\{x\in\mz\mid x\equiv \sum_{i=0}^{n-1} 2^{2^i}\ \mathrm{mod}\ 2^{2^n}\}\]
Clearly $C_1\supset C_2\supset C_3\supset\cdots$ is a descending chain of cosets, hence any finite subsystem intersects non-trivially. However the 
intersection of all the cosets is empty, since the minimum of the absolute values of the elements of $C_n$ tends to infinity, as $n\to\infty$. 
On the other hand, we shall see below that $\kappa(\mr)=3$. }
\end{example} 

Next we show that the Helly dimension of a compact Lie group is finite. We need a variant of Platonov's Lemma in this context. 

\begin{lemma}\label{lemma:platonov} 
Let $G$ be a compact real Lie group, and write 
$G^{\circ}$ for the connected component of the identity. Then $G$ contains a finite subgroup $H$ with $G=HG^{\circ}$.  
\end{lemma} 

\begin{proof} Let $T$ be a maximal torus in $G^{\circ}$.  For any $g\in G$ we have that $gTg^{-1}$ is also a maximal torus in $G^{\circ}$, hence $gTg^{-1}=hTh^{-1}$ for some $h\in G^{\circ}$ (see for example Corollary 4.35 in \cite{knapp}). 
Therefore $G=NG^{\circ}$, where $N$ denotes the normalizer of $T$ in $G$. 
Moreover, the centralizer of $T$ in $G^{\circ}$ is $T$ (see for example Corollary 4.52 in \cite{knapp}), hence the factor group $(N\cap G^{\circ})/T$ is the Weyl group of $G^{\circ}$ (see e.g. Theorem 4.54 in \cite{knapp}). In particular, $(N\cap G^{\circ})/T$ is finite, implying that $Q=N/T$ is finite. So it is sufficient to find a finite subgroup $H$ in $N$ with $N=HT$. For this purpose we modify the proof of the Schur-Zassenhaus Theorem, given for example in 9.1.2 of \cite{robinson}. 
Take a transversal $\{t_x\mid x\in Q\}$ to $T$ in $N$. 
For alll $x,y\in Q$ there is an element $c(x,y)$ of $T$ with 
$t_xt_y=t_{xy}c(x,y)$. Set $d(y):=\prod_{x\in Q}c(x,y)$. 
The cocycle condition on $c$ implies that  
$d(yz)=d(y)^zd(z)c(y,z)^{-q}$ for all $y,z\in Q$, where $q=|Q|$. 
Set $E:=\{u\in T\mid u^q=1\}$. Clearly $E$ is a finite subgroup: 
we have $|E|=q^{\dim(T)}$. Moreover, it is a characteristic subgroup of $T$, 
hence is preserved by conjugation by an arbitrary element of $N$. 
Denote $e(*)$ any $q$th root of $d(*)$ (note that being the direct product of circle groups, $T$ is divisible). 
Then $e(yz)u(y,z)=e(y)^ze(z)c(y,z)$, where $u(y,z)$ is an appropriate element of $E$. Now set $s_x:=t_xe(x)$. We have 
\[s_ys_z=t_yt_ze(y)^ze(z)=t_{yz}c(y,z)e(y)^ze(z)=
t_{yz}e(yz)u(y,z)=s_{yz}u(y,z).\] 
This obviously shows that $E$ and the $s_x$ $(x\in Q)$ generate a finite subgroup $H$ of $N$, and $N=HT$. 
\end{proof} 

\begin{theorem}\label{thm:compact} 
The Helly dimension $\kappa(G)$ of a 
compact real Lie group $G$ is finite. 
\end{theorem} 

\begin{proof} Since $G$ has a faithful finite dimensional smooth representation, we may assume that $G$ is a closed subgroup of the group of invertible $n\times n$ matrices with complex entries. 
Note that any closed subgroup of $G$ is compact. 
Moreover, closed cosets are submanifolds of $G$, hence we may speak about their {\it dimension}. 
Now we may repeat verbatim the proof of Theorem~\ref{thm:alggphelly}. 
Using Lemma~\ref{lemma:platonov} at the appropriate point, the proof reduces to Corollary~\ref{cor:linear}. 
\end{proof} 

\begin{remark}\label{remark:compsupport} {\rm The Helly dimension $\kappa$ of the compact real Lie group $G$ has the following interpretation: Given compact topological $G$-spaces $X_i$  
$(i=1,\ldots,m)$, we have the diagonal action of $G$ on the product space 
$X:=\prod_{i=1}^m X_i$. Consider the orbit space $X/G$ with the factor topology. Write $C(Y)$ for the algebra of real valued continuous functions on a topological space $Y$. For a subset $I\subseteq \{1,\ldots,m\}$, set 
$X_I:=\prod_{i\in I}X_i$, and view 
$C(X_I)$ as a subalgebra of $C(X)$ in the obvious manner. 
By the Stone-Weierstrass Theorem (see e.g. page 2 in \cite{hochschild}) 
the $\mr$-subalgebra of $C(X)$ generated by $\{C(X_i)\mid i=1,\ldots,m\}$ 
is dense in $C(X)$ with respect to the topology on $C(X)$ induced by the maximum norm. On the other hand, in general the subalgebras $\{C(X_i/G)\mid i=1,\ldots,m\}$ do not generate a dense subalgebra in $C(X/G)$ (for example, it may well happen 
that all the $X_i$ are homogeneous spaces, so $C(X_i/G)$ consist of constant functions, but $X/G$ is not just a point). 
Now $\kappa(G)$ is the minimal natural number $d$ such that for any finite collection  
$X_i$ $(i=1,\ldots,m)$ of compact $G$-spaces, the subalgebras $C(X_I/G)$ with  $I\subseteq\{1,\ldots,m\}$, 
$|I|\leq d$ generate a dense subalgebra in $C(X/G)$.  } 
\end{remark}

We conclude by computing the Helly dimension of the additive group of $\mr$. 

\begin{proposition}\label{prop:kappa(R)} 
\begin{itemize}
\item[(i)] We have the equality $\kappa(\mz)=2$. 
\item[(ii)] We have the equality $\kappa(\mr)=3$. 
\end{itemize} 
\end{proposition} 

\begin{proof} (i) Take a finite system of cosets $C_1,\ldots,C_m$ in $\mz$, and assume that any two of them have a common element. 
Denote by $H_1,\ldots,H_m$ the corresponding subgroups. 
If one of them, say $H_1$ is the trivial subgroup, then $C_1$ is a single element, and this element is contained in all the cosets by assumption. 
Assume now that none of the $H_i$ is trivial. Then their intersection is the 
subgroup $n\mz$ for some $n\in\mn$. Consider the natural surjection 
$\eta:\mz\to \mz/n\mz$ onto the cyclic group of order $n$. Since 
$H_i\supseteq \ker(\eta)$ for all $i$, we have $C_i=\eta^{-1}(\eta(C_i))$. 
Now $\eta(C_i)$, $i=1,\ldots,m$ are cosets in the finite cyclic group 
$\mz/n\mz$ such that any two have a common element. 
Since $\kappa(\mz/n\mz)=2$ by Corollary~\ref{cor:abelkappa}, we have 
that $\cap_{i=1}^m\eta(C_i)\neq\emptyset$, hence 
$\cap_{i=1}^mC_i\neq\emptyset$. 

(ii) Consider in the additive group of $\mr$ the following three closed cosets: 
$C_1=\mz$, $C_2=\{n\sqrt 2\mid n\in\mz\}$, $C_3=\{1+n(\sqrt 2-1)\mid n\in\mz\}$. 
Then we have $C_1\cap C_2=\{0\}$, $C_1\cap C_3=\{1\}$, and 
$C_2\cap C_3=\{\sqrt 2\}$. Therefore $\kappa(\mr)\geq 3$. 
For the reverse inequality, let $C_1,\ldots,C_m$ be a finite system 
of closed cosets in $\mr$, where $m\geq 3$. 
Translating the cosets by the opposite of an element of $C_m$ we may assume that 
$C_m$ is a subgroup of $\mr$. We may assume that $C_m\neq \mr$. 
Any closed non-trivial proper subgroup of $\mr$ is isomorphic to $\mz$. 
Now $C_1\cap C_m,\ldots,C_{m-1}\cap C_m$ are pairwise intersecting cosets in 
$C_m\cong\mz$, hence they have a common element by (i). 
\end{proof} 

%%%%%%%%%%%%%%%%%%%%%%%%%%%%%%%%%%%%%%%%%%%%%%%%%%%%%

\begin{center} {\bf Acknowledgements}\end{center} 

We thank L\'aszl\'o Pyber for stimulating discussions and references. 

%%%%%%%%%%%%%%%%%%%%%%%%%%%%%%%%%%%%%%%%%%%%%%%%%%%%%%%%%%%%%%%%


\begin{thebibliography}{MMM}

\bibitem{bass} H. Bass, Theorems of Jordan and Burnside for algebraic groups, 
J. Algebra 82 (1983), 245-254. 

\bibitem{birkes} 
D. Birkes, 
Orbits of linear algebraic groups, 
Ann. Math. II., Ser. 93 (1971), 459-475. 

\bibitem{borel} A. Borel, Linear Algebraic Groups, Second Enlarged Edition, Springer-Verlag, 1991. 

\bibitem{derksen-kemper} 
H. Derksen and G. Kemper, 
Computational Invariant Theory, 
Springer-Verlag, Berlin, 2002. 

\bibitem{d:sep} M. Domokos, Typical separating invariants, Transform. Groups, Vol. 12 (2007), 49-63. 

\bibitem{draisma-kemper-wehlau} J. Draisma, G. Kemper, and D. Wehlau, 
Polarization of separating invariants, Canad. J. Math.  60 (2008), no. 3, 556-571.

\bibitem{dufresne} E. Dufresne, 
Separating invariants and finite reflection groups, Adv. Math. 221 (6) (2009) 1979-1989.  

\bibitem{dufresne-elmer-kohls} E. Dufresne, J. Elmer, M. Kohls, 
The Cohen-Macaulay property of separating invariants of finite groups, 
arXiv:0904.1069. 

\bibitem{grosshans:1997} F. D. Grosshans, Algebraic Homogeneous Spaces and
Invariant Theory, Lecture Notes in Math. 1673 (1997),
Springer, Berlin-Heidelberg-New York. 

\bibitem{grosshans:2007} F. D. Grosshans, 
Vector invariants in arbitrary characteristic, 
Transform. Groups 12 (2007), no. 3, 499-514. 

\bibitem{hochschild} G. Hochschild, The Structure of Lie Groups, 
Holden-Day, Inc., San Francisco, 1965. 

\bibitem{kemper} G. Kemper, Separating invariants,  J. Symbolic Computation 44 (2009), 1212-1222. 

\bibitem{knapp} A. W. Knapp, Lie Groups Beyond an Introduction, 
Birkh\"auser, Basel, 2002. 

\bibitem{neusel-sezer} M. D. Neusel and M. Sezer, 
Separating invariants of modular p-groups and groups acting diagonally, 
Math. Res. Letters, to appear. 


\bibitem{robinson} Derek J. S. Robinson, 
A Course in the Theory of Groups, Springer-Verlag, Berlin, 1982. 

\bibitem{steinitz} J. Eckhoff, 
Helly, Radon, and Carath\'eodory Type Theorems, pp. 389-448 in 
Handbook of Convex Geometry, Vol A (Ed.: P. M. Gruber, J. M. Willy), 
North-Holland, Amsterdam, 1993. 

\bibitem{popov-vinberg} 
V. L. Popov and E. B. Vinberg, 
Invariant theory, in: Algebraic Geometry IV, Encyclopaedia of Mathematical 
Sciences 55, Springer-Verlag, Berlin, Heidelberg, 1994. 

\bibitem{rosenlicht} 
M. Rosenlicht, 
A remark on quotient spaces,  
An. Acad. Brasil. Ci. 35 (1963), 487-489.


\end{thebibliography}
 \end{document}